\documentclass[letterpaper,12pt]{article}
\usepackage{tikz}
\usepackage{graphicx}
\usepackage{color}
\usepackage{setspace}
\usepackage{url}
\usepackage{latexsym,amsmath,amssymb,theorem,graphics,epsfig,subfigure}
\usepackage{color}
\usepackage{multirow}
\theoremstyle{break}
\newtheorem{theorem}{Theorem}
\newtheorem{corollary}{Corollary}
\newtheorem{proposition}{Proposition}
\newtheorem{lemma}{Lemma}
\theorembodyfont{\rm}
\newtheorem{definition}{Definition}
\newtheorem{example}{Example}

\newtheorem{remark}{Remark}

\newcommand{\lran}{\longrightarrow}

\def\lint#1{\left[{#1}\right[}
\def\rint#1{\left]{#1}\right]}
\def\cint#1{\left[{#1}\right]}
\def\opint#1{\left]{#1}\right[}
\def\uint{[0,1]}

\def\ouint{\left]0,1\right[}

\def\prooftxt{\mbox{\large\sc proof: }}
\def\konproof{\rm\hspace*{\fill}$\Box$}
\newenvironment{proof}{\par\smallskip\noindent\prooftxt }%
{\konproof\par\vspace*{6pt}}
\newcounter{inst}
\newenvironment{mylist}{\begin{list}{(\roman{enumi})}{\usecounter{enumi}%
                        \setlength{\listparindent}{0pt}%
                        \setlength{\labelsep}{0.4em}%
                        \setlength{\labelwidth}{2.1em}
                        \setlength{\leftmargin}{2em}%
                        \setlength{\itemsep}{-0.8ex}%
                        \setlength{\topsep}{0.4ex}%
                        \setlength{\rightmargin}{0pt}}}{\end{list}}
\date{}
\author{
{\bf Andrea Mesiarov\' a-Zem\' ankov\' a} \\ \ Mathematical Institute, Slovak Academy of Sciences \\
Bratislava, SLOVAKIA \\
zemankova@mat.savba.sk
}

\title{\bf  Characterization of uninorms with continuous underlying t-norm and t-conorm by their set of  discontinuity points}

\begin{document}

\maketitle \thispagestyle{empty}

\begin{abstract} Uninorms with continuous underlying t-norm and t-conorm are discussed and properties of the set of discontinuity points of such a uninorm are shown. This set  is proved to be a subset of the graph of a special symmetric, surjective,  non-increasing set-valued function.
 A sufficient condition for a uninorm to have continuous underlying operations is also given.
 Several examples are  included.

\end{abstract}

\section{Introduction}
\label{sec1}

The (left-continuous) t-norms and their dual t-conorms have an indispensable role in many domains \cite{haj98,16,sug85}.
 Generalizations of t-norms and t-conorms that can model bipolar behaviour are uninorms
 (see \cite{FYR97,jacon,YR96}). The class of uninorms is widely used both in theory \cite{PT1,uniru} and in applications \cite{PA2,PA1}.
 The complete characterization of uninorms with continuous underlying t-norm and t-conorm has been in the center of the interest for a long time, however, only partial results were achieved (see \cite{D07,ordut,DRT,FB12,HL01,Lili,LLF15,Lilif,MesPet14,MMT03,MMRT15,RTBF10,Rut14}).

 In \cite{genuni} we have introduced ordinal sum of uninorms and in \cite{repord} we have characterized uninorms that are ordinal sums of representable uninorms.
  We would like to characterize all uninorms with continuous underlying functions and obtain a similar representation as in the case of t-norms and t-conorms. In this paper we will show that underlying operations of a uninorm $U$ are continuous if and only if $U$ is continuous on $\uint^2\setminus R,$ where $R$ is the graph of a special symmetric, surjective,  non-increasing set-valued function and $U$ is in each point $(x,y)\in \uint^2$
  either left-continuous or right-continuous, or continuous.
  We will then continue and in \cite{extord,supja} we will show that  each uninorm with continuous underlying t-norm and t-conorm can be decomposed into
  an  ordinal sum of semigroups related to representable uninorms, continuous Archimedean t-norms and t-conorms, internal uninorms and singleton semigroups.

In Section \ref{sec2} we will recall all necessary basic notions and results. We will characterize uninorms with continuous underlying functions via the properties of their set of discontinuity points (Section \ref{sec3}). We give our conclusions in Section \ref{sec5}.

\section{Basic notions and results}
\label{sec2}
   Let us now recall all necessary basic notions.

 A triangular norm  is a  function $T \colon \uint^2 \lran \uint $ which is commutative,
associative, non-decreasing in both variables and $1$ is its neutral element.
 \textcolor{red}{Note that in this paper we stick to the definition from \cite{KMP00}, where a non-decreasing function means an increasing function that
 need not to be strictly increasing.}
 Due to the associativity, $n$-ary form of any t-norm is uniquely given and  thus it can be extended to an aggregation function working on  $\bigcup_{n\in \mathbb{N}}\uint^n.$
Dual functions to t-norms are t-conorms. A triangular conorm  is a  function $S \colon \uint^2 \lran \uint $ which is commutative,
associative, non-decreasing in both variables and $0$ is its neutral element.
The duality between t-norms and t-conorms is expressed by the fact that from any t-norm $T$ we can obtain its dual t-conorm $S$ by the equation
$$S(x,y)=1-T(1-x,1-y)$$
and vice-versa.

\begin{proposition}[\cite{KMP00}]
Let $t\colon \uint\lran \cint{0,\infty}$ ($s\colon \uint\lran \cint{0,\infty}$) be a continuous strictly decreasing (increasing) function such that $t(1)=0$ ($s(0)=0$). Then the   operation  $T \colon \uint^2 \lran \uint $ ( $S \colon \uint^2 \lran \uint $)  given by
$$T(x,y)=t^{-1}(\min(t(0),t(x)+t(y)))$$
$$S(x,y)=s^{-1}(\min(s(1),s(x)+s(y)))$$
is a continuous t-norm (t-conorm). The function $t$ ($s$) is called an \emph{additive generator} of $T$ ($S$).
\end{proposition}

An additive generator of an \textcolor{red}{Archimedean} continuous t-norm $T$ (t-conorm $S$) is uniquely determined up to a positive multiplicative constant.
Each continuous t-norm (t-conorm) is equal to an ordinal sum of continuous Archimedean  t-norms (t-conorms). Note that a continuous t-norm (t-conorm) is Archimedean if and only if it has only trivial idempotent points $0$ and $1.$ A continuous Archimedean t-norm $T$ (t-conorm $S$) is either strict, i.e.,
strictly increasing on $\rint{0,1}^2$ (on $\lint{0,1}^2$), or nilpotent, i.e.,  there exists $(x,y)\in \opint{0,1}^2$ such that $T(x,y)=0$
($S(x,y)=1$). Moreover, each continuous Archimedean t-norm (t-conorm) has a continuous additive generator. More details on t-norms and t-conorms can be found in \cite{AFS06,KMP00}.

A uninorm  (introduced in \cite{YR96}) is a  function  $U\colon \uint^2 \lran \uint $  which is commutative,
associative, non-decreasing in both variables and have a neutral element $e\in \ouint$ (see also \cite{FYR97}). If we take a uninorm in a broader sense, i.e., if for a neutral element we have  $e\in \uint,$ then the class of uninorms covers also the class of t-norms and the class of t-conorms. In order to stress that we assume a uninorm with $e\in \opint{0,1}$ we will call such a uninorm \emph{proper}. For each uninorm the value $U(1,0)\in \{0,1\}$ is the annihilator of $U.$ A uninorm is called \emph{conjunctive} (\emph{disjunctive}) if $U(1,0)=0$ ($U(1,0)=1$). Due to the associativity, we can uniquely define $n$-ary form of any uninorm for any $n\in\mathbb{N}$ and therefore in some proofs we will use ternary form instead of binary, where suitable.

For each uninorm $U$ with the  neutral element $e\in \uint,$ the restriction of $U$ to $\cint{0,e}^2$ is a t-norm on $\cint{0,e}^2,$ i.e., a linear transformation of some t-norm $T_U$  on $\uint^2$
and the restriction of $U$ to $\cint{e,1}^2$ is a t-conorm  on $\cint{e,1}^2,$ i.e., a linear transformation of some t-conorm $S_U$  on  $\uint^2.$ Moreover, $\min(x,y)\leq U(x,y)\leq \max(x,y)$ for all
$(x,y)\in \cint{0,e}\times \cint{e,1}\cup \cint{e,1}\times \cint{0,e}.$ We will denote the set of all uninorms $U$ such that $T_U$ and $S_U$ are continuous by $\mathcal{U}.$

From any pair of a t-norm and a t-conorm we can construct the minimal and the maximal uninorm with the given
underlying functions.

\begin{proposition}[\cite{meszem_LS00}]
Let  $T \colon \uint^2 \lran \uint $ be a t-norm and  $S \colon \uint^2 \lran \uint $ a t-conorm and assume $e\in \uint.$ Then the two functions  $U_{\min},U_{\max} \colon \uint^2 \lran \uint $ given by
$$U_{\min}(x,y)=\begin{cases} e\cdot T(\frac{x}{e},\frac{y}{e}) &\text{if $(x,y)\in \cint{0,e}^2,$ } \\
e+ (1-e)\cdot S(\frac{x-e}{1-e},\frac{y-e}{1-e}) &\text{if $(x,y)\in \cint{e,1}^2,$ } \\
\min(x,y) &\text{otherwise}\end{cases}$$ and
$$U_{\max}(x,y)=\begin{cases} e\cdot T(\frac{x}{e},\frac{y}{e}) &\text{if $(x,y)\in \cint{0,e}^2,$ } \\
e+ (1-e)\cdot S(\frac{x-e}{1-e},\frac{y-e}{1-e}) &\text{if $(x,y)\in \cint{e,1}^2,$ } \\
\max(x,y) &\text{otherwise}\end{cases}$$ are uninorms. We will denote the set of all uninorms of the first type by $\mathcal{U}_{\min}$ and of the second type by $\mathcal{U}_{\max}.$ \label{mmset}
\end{proposition}

%\begin{definition}
%A uninorm  $U \colon \uint^2 \lran \uint $ is called
%\emph{internal} if  $U(x,y)\in \{x,y\}$ for all $(x,y)\in \uint^2.$ Moreover, $U$ is called s-internal if it is internal and there exists a continuous and %strictly decreasing function $v_{U}\colon \uint \lran \uint$ such that  $U(x,y)=\min(x,y)$ if $y<v_U(x)$ and  $U(x,y)=\max(x,y)$ if $y>v_U(x).$
%\end{definition}

%From \cite{B98} we can deduce the following easy result.

%\begin{lemma}
%Let $U\colon \uint^2 \lran \uint $  be a uninorm such that $T_U=\min$ and $C_U=\max.$ Then $U$ is internal. \label{lemmm}
%\end{lemma}

Similarly as in the case of t-norms and t-conorms we can construct uninorms using additive generators (see  \cite{FYR97}).

\begin{proposition}[\cite{FYR97}]
Let $f\colon \uint\lran\cint{-\infty,\infty},$  $f(0)=-\infty,$ $f(1)=\infty$ be a continuous strictly increasing function.
Then a function $U \colon \uint^2 \lran \uint $ given by
$$U(x,y)=f^{-1}(f(x)+f(y)),$$ where $f^{-1}\colon \cint{-\infty,\infty}\lran \uint$ is an inverse function to $f,$
 \textcolor{red}{with the convention $\infty +(-\infty) = \infty$ (or $\infty +(-\infty) = -\infty$),}
 is a uninorm, which will be called a \emph{representable} uninorm.
\end{proposition}

Note that if we relax the strict monotonicity of the additive generator then the neutral element will be lost and by relaxing the condition  $f(0)=-\infty,$ $f(1)=\infty$ the associativity will be lost (if $f(0)<0$ and $f(1)>0$). In  \cite{uniru} (see also \cite{jacon}) we can find the following result.

\begin{proposition}[\cite{uniru}]
Let $U\colon \uint^2 \lran \uint$ be a uninorm continuous everywhere on the unit square except of the two points $(0,1)$ and $(1,0).$ Then $U$ is representable.
\label{prouni}
\end{proposition}

%Thus a uninorm $U$ is representable if and only if it is continuous on $\uint^2\setminus \{(0,1),(1,0)\},$ which
% completely characterizes the set of representable uninorms.

%An ordinal sum of uninorms was introduced in \cite{genuni} and it is based on the ordinal sum of semigroups introduced by Clifford \cite{cli}.
%We will not go into details on ordinal sum of uninorms and in our examples we will use the basic result of Clifford.

For our examples we will use the following ordinal sum construction introduced by Clifford.

\begin{theorem}[\cite{cli}]
Let $A\neq \emptyset$ be a totally ordered set and $(G_{\alpha})_{\alpha\in A}$ with $G_{\alpha}=(X_{\alpha},*_{\alpha})$ be a family of semigroups.
Assume that for all
$\alpha,\beta\in A$ with
$\alpha<\beta$ the sets $X_{\alpha}$
 and $X_{\beta}$ are either disjoint or that $X_{\alpha} \cap X_{\beta}=\{x_{\alpha,\beta}\},$  where
$x_{\alpha,\beta}$ is both the neutral element of $G_{\alpha}$
 and the annihilator of $G_{\beta}$ and where for each $\gamma\in A$ with
$\alpha<\gamma<\beta$ we
have $X_{\gamma}=\{x_{\alpha,\beta}\}.$ Put $X=\bigcup\limits_{\alpha\in A}X_{\alpha}$
 and define the binary operation $*$ on $X$ by
$$x*y=\begin{cases} x*_{\alpha} y &\text{if $(x,y)\in X_{\alpha} \times X_{\alpha},$} \\
x &\text{if $(x,y)\in X_{\alpha} \times X_{\beta}$ and $\alpha<\beta,$} \\
y &\text{if $(x,y)\in X_{\alpha} \times X_{\beta}$ and $\alpha>\beta.$} \\ \end{cases}$$
Then $G =(X,*)$ is a semigroup. The semigroup $G$ is commutative if and only if for each
 $\alpha \in A$ the semigroup
$G_{\alpha}$
 is commutative. \label{thcli}
\end{theorem}

Therefore in our examples the commutativity and the associativity of the corresponding ordinal sum uninorm will follow from Theorem \ref{thcli}. Monotonicity and the neutral element can be then
easily checked by the reader.

Further we will use the following transformation.
For any $0\leq a < b\leq c < d\leq 1,$  $v\in \cint{b,c},$ and a uninorm $U$  with the neutral element $e\in \uint$ let $f\colon \uint \lran \lint{a,b} \cup \{v\} \cup \rint{c,d}$ be
 given by
\begin{equation}f(x)=\begin{cases}  (b-a)\cdot \frac{x}{e} + a  &\text{if $x\in \lint{0,e} ,$ } \\
v  &\text{if $x=e ,$ } \\
d - \frac{(1-x)(d-c)}{(1-e)} &\text{otherwise.} \end{cases}\label{utra}\end{equation}
Then $f$  is linear on $\lint{0,e}$ and on $\rint{e,1}$ and thus it is a
piece-wise linear isomorphism of $\uint$ to $(\lint{a,b} \cup \{v\} \cup \rint{c,d})$  and a  function
$U^{a,b,c,d}_v\colon (\lint{a,b} \cup \{v\} \cup \rint{c,d})^2 \lran (\lint{a,b} \cup \{v\} \cup \rint{c,d})$ given by
\begin{equation}U^{a,b,c,d}_v(x,y)=f(U(f^{-1}(x),f^{-1}(y)))\label{unitr}\end{equation} is \textcolor{red}{an operation on $(\lint{a,b} \cup \{v\} \cup \rint{c,d})^2$ which is commutative, associative, non-decreasing in both variables (with respect to the standard order) and $v$ is its neutral element.}

\begin{example}
 Assume $U_1 \in \mathcal{U}_{\min}$ and $U_2\in \mathcal{U}_{\max}$ with respective neutral elements $e_1,e_2.$
Then $U_1$  is an ordinal sum of semigroups $G_{\alpha} = (\lint{0,e},T^*_{U_1})$ and $G_{\beta} = (\cint{e,1},S^*_{U_1})$ with
$\alpha < \beta,$ where $T^*_{U_1} = U_1\vert_{\cint{0,e_1}^2}$ and $S^*_{U_1} = U_1\vert_{\cint{e_1,1}^2}.$  Similarly,  $U_2$  is an ordinal sum of semigroups $G_{\alpha} = (\cint{0,e},T^*_{U_2})$ and $G_{\beta} = (\rint{e,1},S^*_{U_2})$ with
$\alpha > \beta.$
 If all underlying operations are continuous then the set of discontinuity points of $U_1$ is equal to the set $S_1=\{e\} \times \rint{e,1} \cup \rint{e,1} \times \{e\}$ and the set of discontinuity points of $U_2$ is equal to the set
$S_2=\{e\} \times \lint{0,e} \cup \lint{0,e} \times \{e\}.$ Both uninorms can be seen on Figure \ref{figst1}.
 \label{examplo}
\end{example}

\begin{figure}[htb]  \hskip2cm
\begin{picture}(150,150)
\put(0,0){ \line(1,0){150} }
\put(0,0){ \line(0,1){150} }
\put(150,150){ \line(-1,0){150} }
\put(150,150){ \line(0,-1){150} }
\put(0,75){ \line(1,0){150} }
\put(75,0){ \line(0,1){150} }
\linethickness{0.5mm}
\put(77,75){ \line(1,0){73} }
\put(75,77){ \line(0,1){73} }
\put(79,75){\circle{4}}

\put(30,30){$T_{U_1}^*$}
\put(110,110){$S_{U_1}^*$}
\put(110,30){$\min$}
\put(30,110){$\min$}
\end{picture}
\hskip2cm
\begin{picture}(150,150)
\put(0,0){ \line(1,0){150} }
\put(0,0){ \line(0,1){150} }
\put(150,150){ \line(-1,0){150} }
\put(150,150){ \line(0,-1){150} }
\put(0,75){ \line(1,0){150} }
\put(75,0){ \line(0,1){150} }
\linethickness{0.5mm}
\put(0,75){ \line(1,0){73} }
\put(75,0){ \line(0,1){73} }
\put(79,75){\circle{4}}

\put(30,30){$T_{U_2}^*$}
\put(110,110){$S_{U_2}^*$}
\put(110,30){$\max$}
\put(30,110){$\max$}
\end{picture}
\caption{The uninorm $U_1$ (left) and the uninorm $U_2$ (right) from Example \ref{examplo}.  The bold lines denote the points of discontinuity of $U_1$ and $U_2.$} \label{figst1}
\end{figure}
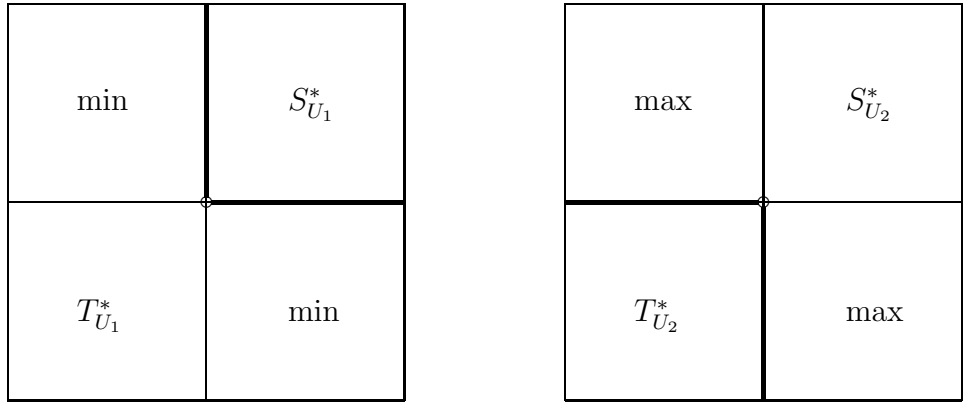

More detailed discussion on the ordinal sum construction for uninorms can be found in \cite{extord}.

\section{Characterization of uninorms $U\in \mathcal{U}$ by means of  special set-valued functions}
\label{sec3}

In this section we will show  that for a uninorm $U$ we have $U\in \mathcal{U}$ if and only if $U$ is continuous on
$\uint^2\setminus R,$ where $R$ is the graph of a special symmetric, surjective,  non-increasing  set-valued function $r$ and $U$ is in each point $(x,y)\in \uint^2$ either left-continuous, or right-continuous, or continuous.
In the first part we will focus on the necessity part, i.e., we will show that each  uninorm $U\in \mathcal{U}$ is continuous on
$\uint^2\setminus R,$ where $R$ is the graph of  some symmetric, surjective,  non-increasing  set-valued function $r$ (Theorem \ref{mufu}).
We will also show that $U\in \mathcal{U}$ implies that $U$ is in each point $(x,y)\in \uint^2$ either left-continuous, or right-continuous, or continuous (Theorem \ref{thco}).

\subsection{The necessity part}
\label{sec3.1}

The following lemmas and propositions are necessary for the proof of Theorem \ref{mufu} and \ref{thco}.

\begin{lemma}[\cite{repord}]
Each uninorm $U\colon \uint^2\lran \uint,$ $U\in \mathcal{U},$ is continuous in $(e,e).$
\end{lemma}

Next we show that for $x,y\in \uint$ we have $U(x,y)=\min(x,y)$ or $U(x,y)=\max(x,y)$  if $x$ is an idempotent element of $U.$

\begin{lemma}
Let  $U\colon \uint^2\lran \uint$ be a uninorm and let $U\in \mathcal{U}.$ If $a\in \uint$ is an idempotent point of $U$ then
$U$ is internal on $\{a\}\times \uint,$ i.e., \textcolor{red}{$U(a,x)\in \{x,a\}$ for all $x\in \uint.$}
 \end{lemma}

\begin{proof}
If $a=e$ the result is obvious. Suppose $a<e$ (the case when $a>e$ is analogous). Since $T_U$ is continuous we have
$U(a,x)=\min(a,x)$ if $x\in\cint{0,e}.$ Suppose that there exists $y\in \rint{e,1}$ such that $U(a,y)=c\in \opint{a,y}.$
Then $U(a,c)=U(a,a,y)=U(a,y)=c$ and if $c\leq e$
then $c=U(a,c)\leq a$ what is a contradiction. Thus $y>c>e.$ Then since $S_U$ is continuous there exists a $y_1$ such that $U(c,y_1)=y.$
Then, however, $$U(a,y)=U(a,c,y_1)=U(c,y_1)=y$$ what is again a contradiction. Thus $U$ is internal on $\{a\}\times \uint.$
\end{proof}

 For a given uninorm $U\colon \uint^2\lran \uint$ and  each $x\in \uint$ we define a function $u_x\colon \uint\lran \uint$ by $u_x(z)=U(x,z)$ for $z\in \uint.$

\begin{lemma}
Let $U\colon \uint^2\lran \uint$ be a uninorm, $U\in \mathcal{U},$ and assume  $x\in \uint.$ The function $u_x$ is continuous if and only if one of the following conditions:
\begin{mylist}
\item  $u_x(1)<e,$
\item $u_x(0)>e,$
\item $e\in \mathrm{Ran}(u_x)$\end{mylist}
   is satisfied. \label{lemconux}
\end{lemma}

\begin{proof}
If $e\in \mathrm{Ran}(u_x)$ then there exists a $y\in \uint$ such that $U(x,y)=e.$ Since $U$ is monotone continuity of $u_x$ is equivalent with the equality $ \mathrm{Ran}(u_x) =\cint{a,b}$ for some $a=U(0,x)$ and
$b=U(1,x).$  Assume $c\in \uint.$
Then $U(x,y,c)=c$ and for $z=U(y,c)$ we have $u_x(z)=c,$ i.e.,   $\mathrm{Ran}(u_x)=$ $\uint.$  If $u_x(1)=v<e$ (the case when $u_x(0)>e$ can be shown similarly)
then due to the monotonicity the continuity of $u_x$ is equivalent with the equality $\mathrm{Ran}(u_x)=\cint{0,v}.$ Assume $w\in \cint{0,v}.$  Since $T_U$ is continuous there exists a $q\in \cint{0,e}$ such that $U(v,q)=w,$ i.e., $U(x,1,q)=w$
 and then $u_x(U(1,q))=w.$ Therefore  $\mathrm{Ran}(u_x)=\cint{0,v}.$

 Vice-versa, if $u_x$ is continuous and   $u_x(0)\leq e \leq u_x(1)$ then evidently $e\in \mathrm{Ran}(u_x).$

\end{proof}

\begin{example}
For a representable uninorm $U$ the function $u_x$ is continuous for all $x\in \opint{0,1}.$ If $U$ is conjunctive (disjunctive) then
$u_0$ ($u_1$) is continuous  and $u_1$ ($u_0$) is non-continuous in $0$ ($1$). For a uninorm $U\in \mathcal{U}_{\max}$ ($U\in \mathcal{U}_{\min}$)
 $u_x$ is continuous for all $x\in \cint{e,1}$ ($x\in \cint{0,e}$) and  $u_x$ is non-continuous in $e$ for all $x\in \lint{0,e}$ ($x\in \rint{e,1}$).
\end{example}

Now we recall a result   \cite[Proposition 1]{KD69} which shows a connection between continuity on cuts and joint continuity of a monotone function.

\begin{proposition}
Let $f(x,y)$ be a real valued function defined on an open set $G$ in the plane. Suppose that $f(x,y)$ is
continuous in $x$ and $y$ separately and is monotone in $x$ for each $y.$ Then $f(x,y)$ is (jointly) continuous on the set $G.$ \label{KDpro}
\end{proposition}

The following result shows that if $U(a,b)=e$ then $U$ is continuous in  \textcolor{red}{the point} $(a,b).$
\textcolor{red}{First, however, we introduce two useful lemmas.}

\textcolor{red}{\begin{lemma} 
Let $U\colon \uint^2\lran \uint$ be a uninorm with the neutral element $e\in \uint.$ Then if $U(a,b)=e,$ for some $a,b\in \uint,$ there is either
$a=b=e,$ or $a$ and $b$ are not idempotent elements of $U.$ \label{noid}
\end{lemma}}

\textcolor{red}{\begin{proof}  If $a$ is an idempotent point (similarly for $b$)
then $$e=U(a,b)=U(a,U(a,b))=U(a,e)=a,$$ and $$e=U(a,b)=U(e,b)=b,$$ i.e., $a=b=e.$ \end{proof}}

\textcolor{red}{\begin{lemma}
Let $U\colon \uint^2\lran \uint$ be a uninorm with the neutral element $e\in \uint.$ Then if $U(a,b)=e,$ for some $a,b\in \uint,$ there is either
$a=b=e,$ or $a<e,$  $b>e,$ or $a>e,$ $b<e.$ \label{eeq}
\end{lemma}}

\textcolor{red}{\begin{proof}  If $a=e$ then evidently also $b=e.$ If $a<e$ then $b\neq e$ and we have 
$$e=U(a,b)\leq U(e,b)=b,$$ i.e., $e<b.$ Finally, if $a>e$ then $b\neq e$ and we have
$$e=U(a,b)\geq U(e,b)=b,$$ i.e., $e>b.$ 
 \end{proof}}

\begin{proposition}
Let $U\colon \uint^2\lran \uint$ be a uninorm, $U\in \mathcal{U}.$ If $U(a,b)=e$ for some $a,b\in \uint,$ $a<e,$ then
$U$ is continuous on $\uint^2\setminus (\lint{0,a}\cup\rint{b,1})^2.$ \label{procon}
\end{proposition}

\begin{proof} \textcolor{red}{Since $a<e$ Lemma \ref{noid} implies that $a$ and $b$ are not idempotent elements of $U$ 
 and Lemma \ref{eeq} implies that $b>e.$}
 From Lemma \ref{lemconux} we know that $u_a$ and $u_b$ are continuous functions.
Next we will show that for all $f\in \opint{a,b}$ there exists a $v^f\in \uint$ such that $U(f,v^f)=e.$
Assume $f\in \rint{a,e}$ (for $f\in \lint{e,b}$ the proof is analogous).
 Since $T_U$ is continuous and $U(a,f)\leq a,$ $U(f,e)=f$ there exists an $a^f\in \cint{0,e}$ such that
 $U(f,a^f)=a.$ Then $$e=U(a,b)=U(f,a^f,b)$$ and if $v^f=U(a^f,b)$ then $U(f,v^f)=e.$ Summarising, we get that for all $x\in \cint{a,b}$ the function $u_x$ is continuous. Now since $a$ and $b$ are not idempotents we have $U(a,a)=p<a,$ $U(b,b)=q>b$ and $U(a,a,b,b)=e.$ Thus also all $u_x$ for $x\in \cint{p,q}$ are continuous and then Proposition \ref{KDpro} implies the result.
\end{proof}

\textcolor{red}{
\begin{remark}
If  $U\in \mathcal{U}$ then either $U(x,y)=e$ implies $x=y=e,$ or there exists an $x\neq e$ such that $U(x,y)=e$ for some $y\in \uint.$ We will focus on the second case. Then Lemma \ref{eeq} implies that either 
 $x<e,$ $y>e,$ or $x>e,$  $y<e.$ We will suppose that $x<e$ and $y>e$ (as the other case is analogous). Then associativity implies $U(\underbrace{x,\ldots,x}_{n\text{-times}},\underbrace{y,\ldots,y}_{n\text{-times}})=e$ for all
$n\in \mathbb{N}$ and similarly as in the proof of Proposition \ref{procon} we can show that for all $z\in \cint{U(\underbrace{x,\ldots,x}_{n\text{-times}}),U(\underbrace{y,\ldots,y}_{n\text{-times}})}$ there exists a $q\in \uint$
such that $U(z,q)=e.$ Further, if $U(b,c)=e$ for some $b,c\in \uint,$ $b\neq e,$ then by Lemma \ref{noid} the points $b$ and $c$ are not idempotents. Therefore, in this case,
$U(x,y)=e$ for some $y\in \uint$ if and only if $x\in \opint{a,d},$ where $$a=\lim\limits_{n\lran \infty} U(\underbrace{x,\ldots,x}_{n\text{-times}})$$
and  $$d=\lim\limits_{n\lran \infty} U(\underbrace{y,\ldots,y}_{n\text{-times}}).$$ Note that $a$ and $d$ are idempotent elements of $U$ which follows from the continuity of the underlying functions of $U.$ Further, the monotonicity of $U$ implies that $a<e<d.$ The commutativity of $U$ then implies that
if $U(x,y)=e$ for some $x,y\in \uint$ then $x,y \in \opint{a,d}.$ Vice versa, for all $x \in \opint{a,d}$ there exists a $y \in \opint{a,d}$ such that
$U(x,y)=e.$ Due to Proposition \ref{procon} we see that $U$ is continuous on $\{x\}\times \uint$ for all $x \in \opint{a,d}.$ If we take the union over all
$x\in \opint{a,d}$ then the commutativity of $U$ and Proposition \ref{KDpro} implies that $U$ is continuous on $\opint{a,d}\times \uint \cup \uint \times \opint{a,d}.$
In order to include also the case when $U(x,y)=e$ implies $x=y=e,$ we can generally say that for an $x\in \uint$ there exists a $y\in \uint$ such that $U(x,y)=e$ if and only if $x\in \opint{a,d}\cup \{e\}.$ Note that in the case when $U(x,y)=e$ implies $x=y=e,$ we take $a=e=d.$
\label{remad}
\end{remark}
}

\begin{example}
Assume two representable uninorms $U_1,U_2\colon \uint^2\lran \uint$ with respective neutral elements $e_1,$ $e_2.$
Let $U_1^*$ be a  transformation of $U_1$ to $(\lint{0,\frac{1}{3}}\cup \{v\} \cup \rint{\frac{2}{3},1})^2$ \textcolor{red}{given by \eqref{unitr},} where $v=\frac{1}{3}$ ($v=\frac{2}{3}$) if $U_2$ is conjunctive (disjunctive), and let $U_2^*$ be a linear transformation of $U_2$ to $\cint{\frac{1}{3},\frac{2}{3}}^2.$
Then the ordinal sum of semigroups $G_{\alpha} = (\lint{0,\frac{1}{3}}\cup \{v\} \cup \rint{\frac{2}{3},1},U_1^*),$
$G_{\beta}= (\cint{\frac{1}{3},\frac{2}{3}},U_2^*),$ with $\alpha < \beta,$
is a semigroup $(\uint,U),$ where $U$ is a uninorm with the neutral element $e=\frac{e_2+1}{3}.$
We can find the structure of $U$ on  Figure \ref{figst}. All points of discontinuity of $U$ except $(0,1),(0,1)$ correspond to the transformation of the points $(x,y)\in \uint^2$ such that $U_1(x,y)=e_1.$ For simplicity, we will assume that $U_1(x,1-x)=e_1=\frac{1}{2}$ for all $x\in \opint{0,1}.$ Moreover, for every $a\in \opint{\frac{1}{3},\frac{2}{3}}$ there exists a $b\in \opint{\frac{1}{3},\frac{2}{3}}$ such that $U(a,b)=e.$ The previous result then implies that $U$ is continuous in every point from $\opint{\frac{1}{3},\frac{2}{3}}\times \uint$ and from $\uint\times \opint{\frac{1}{3},\frac{2}{3}}.$
\label{examploa}
\end{example}

\begin{figure}[htb] \begin{center}
\begin{picture}(150,150)
\put(0,0){ \line(1,0){150} }
\put(0,0){ \line(0,1){150} }
\put(150,150){ \line(-1,0){150} }
\put(150,150){ \line(0,-1){150} }

\put(0,50){ \line(1,0){150} }
\put(50,0){ \line(0,1){150} }
\put(0,100){ \line(1,0){150} }
\put(100,0){ \line(0,1){150} }

\put(0,150){ \line(1,-1){50} }
\put(100,50){ \line(1,-1){50} }

\put(70,70){$U_2^*$}
\put(20,20){$U_1^*$}
\put(120,120){$U_1^*$}
\put(70,120){$\max$}
\put(120,70){$\max$}
\put(70,20){$\min$}
\put(20,70){$\min$}
\put(15,110){$U_1^*$}
\put(30,130){$U_1^*$}
\put(115,10){$U_1^*$}
\put(130,30){$U_1^*$}
\end{picture} \end{center}
\caption{The uninorm $U$ from Example \ref{examploa}. The oblique lines denote the points of discontinuity of $U.$} \label{figst}
\end{figure}

In the following results we will continue to investigate properties of the function $u_x.$

\begin{proposition}
Let $U\colon \uint^2\lran \uint$ be a uninorm, $U\in \mathcal{U}.$ Then for each $x\in \uint$ there is at most one point of discontinuity of $u_x.$
Further, if $u_x$ is non-continuous in $y\in \uint$ then $U(x,z)<e$ for all $z<y$ and $U(x,z)>e$ for all $z>y.$ \label{atmost}
\end{proposition}

\textcolor{red}{
\begin{proof} If $u_x$ is non-continuous then Lemma \ref{lemconux} implies $e\notin \mathrm{Ran}(u_x),$ $u_x(0)<e$ and $u_x(1)>e.$
 We will denote $$f=\sup\{U(x,y)\mid y\in \uint,U(x,y)\leq e\}$$ and $$g=\inf\{U(x,y)\mid y\in \uint,U(x,y)\geq e\}.$$ Note that the
 inequality $u_x(0)<e$ ($u_x(1)>e$) implies that $f$ is the supremum ($g$ is the infimum) of a non-empty set.
 Fix arbitrary $f_1 < f.$ Then there exist an $s > 0$ and $y_f$
such that $f_1 \leq f - s \leq U(x, y_f) \leq f < e$ because $f$ is the supremum. Since
$U(U(x,y_f), 0) = 0,$ $U(U(x,y_f),e) = U(x,y_f)$ and $T_U$ is continuous, there exists
an $f_3$ such that $U(U(x,y_f),f_3) = f_1.$  Therefore $U(x,U(y_f,f_3))=f_1$ and
$f_1\in \mathrm{Ran}(u_x).$
\newline
 Similarly, for each $g_1>g$ there is  $g_1\in \mathrm{Ran}(u_x).$ Therefore $\uint\setminus \mathrm{Ran}(u_x)$ is a connected set. Since $u_x$  is monotone it has only one point of discontinuity. Also,  if $u_x$ is non-continuous in $y\in \uint$ then $U(x,z)<e$ for all $z<y$ and $U(x,z)>e$ for all $z>y.$
\end{proof}}

\begin{proposition}
Let $U\colon \uint^2\lran \uint$ be a uninorm, $U\in \mathcal{U}.$  Then for all $x\in \uint$ the function
$u_x$ is either left-continuous or right-continuous. \label{rl}
\end{proposition}

\begin{proof} Assume $x\in \uint.$
From Proposition \ref{atmost} we know that $u_x$ is non-continuous in at most one point, and thus
we will suppose that $u_x$ is non-continuous in the point $p\in \uint.$
  Further, from the proof of the same proposition we know that  $\uint\setminus \mathrm{Ran}(u_x)$ is a connected set, i.e., an interval $I,$ and $u_x(p)$ is an end point of the interval $I.$ Then it is evident that if $u_x(p) = \inf I$ then
$u_x$ is left-continuous and $u_x(p)<e,$  and if $u_x(p) = \sup I$ then $u_x$ is right-continuous and $u_x(p)>e.$
\end{proof}

\begin{remark}
From the proof of Proposition \ref{rl} we see that if $u_x(p)<e$ for some $p\in \uint$ then $u_x$ is left-continuous on $\cint{0,p}$
and   if $u_x(p)>e$  then $u_x$ is right-continuous on $\cint{p,1}.$ \label{remlrint}
\end{remark}

Next we will show that the points of discontinuity of $u_x$ are non-increasing with respect to $x\in \uint.$

\begin{proposition}
Let $U\colon \uint^2\lran \uint$ be a uninorm, $U\in \mathcal{U}.$ Suppose that for $x,x_1\in \uint,$ $x_1<x,$ the functions $u_{x}$ and $u_{x_1}$  are non-continuous in points $y$ and $y_1,$ respectively. Then  $y_1\geq y.$
%Then if for  $x\in \uint$ the function of $u_x$ is non-continuous in $y$ then for all $x_1<x$ the function
%$u_{x_1}$ is either continuous or non-continuous in $y_1$ such that $y_1\geq y.$
\label{promon}
\end{proposition}

\begin{proof}
From the proof of Proposition \ref{atmost} we see that if $u_x$ is non-continuous in $y$ then $U(x,z)<e$
 for all $z<y$ and $U(x,z)>e$
 for all $z>y.$ If $u_{x_1}$ is non-continuous in $y_1$ then the monotonicity implies $U(x_1,z)<e$ for $z<y$ and thus
 $y_1\geq y.$
 \end{proof}

 \begin{corollary}
Let $U\colon \uint^2\lran \uint$ be a uninorm, $U\in \mathcal{U}.$ If  $u_{x_1}$ is non-continuous in $y$ and $u_{x_2}$ is non-continuous in $y$ for some $x_1<x_2$ then
$u_x$ is non-continuous in $y$ for all $x\in \cint{x_1,x_2}.$
\end{corollary}

\begin{proof} Assume $x\in \opint{x_1,x_2}.$
Since $u_{x_1}$  is non-continuous in $y$ we have
 $U(x_1,z)>e$ for all $z>y$ and the monotonicity gives  $U(x,z)>e$ for all $z>y.$ Since $u_{x_2}$  is non-continuous in $y$ we have
 $U(x_2,z)<e$ for all $z<y$ and the monotonicity gives $U(x,z)<e$ for all $z<y.$ Thus $u_x$ is either non-continuous in $y$ or
 $U(x,y)=e.$  
\textcolor{red}{Assume that $U(x,y)=e.$ If $x=y=e$ then $x_1<e<x_2$ and we get $$e<U(x_1,x_2)<e$$ what is a contradiction.} 
   Therefore \textcolor{red}{by Lemma \ref{noid} the points $x$ and $y$ are not idempotent elements of $U$ and $x\neq e,$ $y\neq e.$
 \newline
   Suppose that $y>e.$
   Then 
   $U(x,x,y,y)=e$ with $U(x_1,y,y)>e$ implies $U(x,x)<x_1<x$ and 
  by Proposition \ref{procon}  $U$ is continuous on $\uint^2 \setminus (\lint{0,U(x,x)}\cup \rint{U(y,y),1})^2.$ Then, however, $u_{x_1}$ is continuous,
   what is a contradiction.
  \newline In the case when $y<e$ then $U(x,x,y,y)=e$ with $U(x_2,y,y)<e$ implies 
   $x<x_2<U(x,x)$ and
using Proposition \ref{procon} again we obtain that  $u_{x_2}$ is continuous,
   what is a contradiction.
       Therefore in both cases $U(x,y)\neq e$ and thus $u_x$ is  non-continuous in $y.$ }
   \end{proof}

\begin{example}
Assume a representable uninorm $U_1 \colon\uint^2\lran \uint$ with the neutral element $e_1$ and
a uninorm $U_2\in \mathcal{U}_{\max}$  with the neutral element $e_2=\frac{1}{2}.$
Let $U_1^*$ be a  transformation of $U_1$ to $(\lint{0,\frac{1}{3}}\cup \cint{\frac{2}{3},1})^2$ given by \eqref{unitr},
 and let $U_2^*$ be a linear transformation of $U_2$ to $\cint{\frac{1}{3},\frac{2}{3}}^2.$
Then the ordinal sum of semigroups $G_{\alpha} = (\lint{0,\frac{1}{3}}\cup\cint{\frac{2}{3},1},U_1^*),$
$G_{\beta}= (\cint{\frac{1}{3},\frac{2}{3}},U_2^*),$ with $\alpha < \beta,$
is a semigroup $(\uint,U),$ where $U$ is a uninorm.
We can find the structure of $U$ on  Figure \ref{figst2}. Here $u_{\frac{1}{2}}$ is continuous and $u_0$ ($u_1$) is continuous if $U_1$ is conjunctive (disjunctive). In all other cases $u_x$ is non-continuous. Further, $u_{\frac{1}{3}}$ is non-continuous in
$e=\frac{1}{2}$ and $u_{\frac{2}{3}}$ is non-continuous in
$\frac{1}{3}.$

\label{exampc}
\end{example}

\begin{figure}[htb] \begin{center}
\begin{picture}(150,150)
\put(0,0){ \line(1,0){150} }
\put(0,0){ \line(0,1){150} }
\put(150,150){ \line(-1,0){150} }
\put(150,150){ \line(0,-1){150} }

\put(0,50){ \line(1,0){150} }
\put(50,0){ \line(0,1){150} }
\put(0,100){ \line(1,0){150} }
\put(100,0){ \line(0,1){150} }

\put(0,150){ \line(1,-1){50} }
\put(100,50){ \line(1,-1){50} }

\put(75,50){ \line(0,1){50} }
\put(50,75){ \line(1,0){50} }
\linethickness{0.5mm}
\put(75,50){ \line(0,1){23} }
\put(50,75){ \line(1,0){23} }
\put(50,75){ \line(0,1){25} }
\put(75,50){ \line(1,0){25} }
\put(79,75){\circle{4}}

\put(60,60){$U_2^*$}
\put(85,85){$U_2^*$}
\put(56,80){$\max$}
\put(82,60){$\max$}

\put(20,20){$U_1^*$}
\put(120,120){$U_1^*$}
\put(70,120){$\max$}
\put(120,70){$\max$}
\put(70,20){$\min$}
\put(20,70){$\min$}
\put(15,110){$U_1^*$}
\put(30,130){$U_1^*$}
\put(115,10){$U_1^*$}
\put(130,30){$U_1^*$}
\end{picture} \end{center}
\caption{The uninorm $U$ from Example \ref{exampc}. The oblique and bold lines denote the points of discontinuity of $U.$} \label{figst2}
\end{figure}

Now we will show how can be a point of discontinuity of a uninorm $U$ related to the non-continuity of corresponding functions $u_x.$

\textcolor{red}{
\begin{proposition} Let $U\colon \uint^2\lran \uint$ be a uninorm, $U\in \mathcal{U}.$
 Then $U$ is non-continuous in $(x_0,y_0)\in \uint^2,$ $(x_0,y_0)\neq (e,e),$ if and only if one of the following is satisfied
  \begin{mylist} \item $u_{x_0}$ is non-continuous in $y_0,$ \item $u_{y_0}$ is non-continuous in $x_0,$ \item
  there exist $\varepsilon_1>0$ and $\varepsilon_2>0$ such that $u_z$ is non-continuous in $x_0$ and $u_v$ is non-continuous in $y_0$
 either for all $z\in \rint{y_0,y_0+\varepsilon_1},$ $v\in \rint{x_0,x_0+\varepsilon_2},$  or for all  $z\in \lint{y_0-\varepsilon_1,y_0},$ $v\in \lint{x_0-\varepsilon_2,x_0}.$ \end{mylist}  \label{prononco}
\end{proposition}}

\begin{proof} \textcolor{red}{Suppose that  $U$ is non-continuous in $(x_0,y_0)\in \uint^2.$ Then due to Proposition \ref{procon} $U(x_0,y_0)\neq e.$ }
Since $T_U$ and $S_U$ are continuous we have $(x_0,y_0) \in \cint{0,e}\times \cint{e,1} \cup \cint{e,1}\times \cint{0,e}. $ We will assume $(x_0,y_0) \in \cint{0,e}\times \cint{e,1}$ (the other case is analogous).
From Proposition \ref{KDpro} it follows that if $U$ is non-continuous in $(x_0,y_0)\in \uint^2$ then for all $\delta_1>0$ and all $\delta_2>0$ there exist $x\in \opint{x_0-\delta_1,x_0+\delta_1}$ and  $y\in \opint{y_0-\delta_2,y_0+\delta_2}$ such that
either $u_x$ is  non-continuous in $y$ or $u_y$ is non-continuous in $x.$ Thus $U$ on $\cint{x_0-\delta_1,x_0+\delta_1}\times \cint{y_0-\delta_2,y_0+\delta_2}$ attain values smaller than $e$ and bigger than $e$ as well. Let $W$ be a  subset of $\uint^2$ such that $(x,y)\in W$ if $U(x_1,y_1)<e$ for all $x_1<x,y_1<y$  and $U(x_2,y_2)>e$ for all $x_2>x,y_2>y.$ Then the set $\cint{x_0-\delta_1,x_0+\delta_1}\times \cint{y_0-\delta_2,y_0+\delta_2}\cap W$ is non-empty for all  $\delta_1>0$ and all $\delta_2>0.$ Thus
$(x_0,y_0)\in W.$

\

If $u_{x_0}$ is continuous in $y_0$ then there exists an $\varepsilon_1>0$ such that  either
$u_{x_0}(z)<e$ for all $z\in \cint{y_0-\varepsilon_1,y_0+\varepsilon_1}$ or $u_{x_0}(z)>e$ for all $z\in \cint{y_0-\varepsilon_1,y_0+\varepsilon_1}.$ Similarly, if $u_{y_0}$ is continuous in $x_0$ then there exists an $\varepsilon_2>0$ such that  either
$u_{y_0}(v)<e$ for all $v\in \cint{x_0-\varepsilon_2,x_0+\varepsilon_2}$ or $u_{y_0}(v)>e$ for all $v\in \cint{x_0-\varepsilon_2,x_0+\varepsilon_2}.$ Since we cannot have both $U(x_0,y_0)<e$ and $U(x_0,y_0)>e$ we have either $u_{y_0}(v)<e$ and $u_{x_0}(z)<e$ for all $z\in \cint{y_0-\varepsilon_1,y_0+\varepsilon_1}$ and all $v\in \cint{x_0-\varepsilon_2,x_0+\varepsilon_2},$ or  $u_{y_0}(v)>e$ and $u_{x_0}(z)>e$ for all $z\in \cint{y_0-\varepsilon_1,y_0+\varepsilon_1}$ and all $v\in \cint{x_0-\varepsilon_2,x_0+\varepsilon_2}.$ As these two cases are analogous we will assume
$$u_{y_0}(v)<e \text{ and } u_{x_0}(z)<e \text{ for all } z\in \cint{y_0-\varepsilon_1,y_0+\varepsilon_1} \text{ and all } v\in \cint{x_0-\varepsilon_2,x_0+\varepsilon_2}.$$ Then $U(x_0,y)<e$ for  $ y\in \cint{y_0-\varepsilon_1,y_0+\varepsilon_1}$ and
$U(x,y_0)<e$ for $x\in \cint{x_0-\varepsilon_2,x_0+\varepsilon_2}.$ However, \textcolor{red}{since $(x_0,y_0)\in W,$}  $U(f,g)>e$ for all $f>x_0,g> y_0.$
Thus $u_z$ is non-continuous in $x_0$ and $u_v$ is non-continuous in $y_0$
 for all $z\in \rint{y_0,y_0+\varepsilon_1},$ $v\in \rint{x_0,x_0+\varepsilon_2}.$
 
\textcolor{red}{Vice versa, if  $u_{x_0}$ is non-continuous in $y_0,$ or if $u_{y_0}$ is non-continuous in $x_0,$ then evidently $U$ is non-continuous in $(x_0,y_0).$ Suppose that  there exist $\varepsilon_1>0$ and $\varepsilon_2>0$ such that $u_z$ is non-continuous in $x_0$ and $u_v$ is non-continuous in $y_0$
 either for all $z\in \rint{y_0,y_0+\varepsilon_1},$ $v\in \rint{x_0,x_0+\varepsilon_2},$  or for all  $z\in \lint{y_0-\varepsilon_1,y_0},$ $v\in \lint{x_0-\varepsilon_2,x_0}.$ Then $(x_0,y_0)\in W$ and either $U(x_0,y_0)=e,$ or $U$ is non-continuous in $(x_0,y_0).$ However, if $U(x_0,y_0)=e$ then since $(x_0,y_0)\neq (e,e)$ Lemma \ref{noid} implies that $x_0$ and $y_0$ are not idempotents and Proposition \ref{procon} implies that  $U$ is continuous on $\uint^2\setminus (\lint{0,U(x_0,x_0)}\cup\rint{U(y_0,y_0),1})^2$ if $x_0<e<y_0$ and on  $\uint^2\setminus (\lint{0,U(y_0,y_0)}\cup\rint{U(x_0,x_0),1})^2$ if $x_0>e>y_0.$ In both cases we obtain a contradiction with   the non-continuity of $u_z$ and $u_v.$ Therefore $U(x_0,y_0)\neq e$ and thus  $U$ is non-continuous in $(x_0,y_0).$  }
\end{proof}

\begin{example}
Assume two t-norms $T_1,T_2\colon \uint^2\lran \uint,$ such that $T_2$ has no zero divisors,
 and a t-conorm $S\colon \uint^2\lran \uint.$
Let $T_1^*$ ($T_2^*$) be a linear transformation of $T_1$ ($T_2$) to $\cint{0,\frac{1}{3}}^2$ ($\cint{\frac{1}{3},\frac{2}{3}}^2$),
 and let $S_2^*$ be a linear transformation of $S_2$ to $\cint{\frac{2}{3},1}^2.$
Then the ordinal sum of semigroups $G_{\alpha} = (\cint{0,\frac{1}{3}},T_1^*),$
$G_{\beta}= (\cint{\frac{1}{3},\frac{2}{3}},T_2^*),$
$G_{\gamma}= (\rint{\frac{2}{3},1},S_2^*),$
 with $\alpha < \gamma < \beta,$
is a semigroup $(\uint,U),$ where $U$ is a uninorm (see Figure \ref{figex}).
    If we define an operation $V\colon \uint^2\lran \uint$ by
    $$V = \begin{cases} \min(x,y) &\text{if $x=\frac{1}{3}, y\in \cint{\frac{2}{3},1},$} \\
    \min(x,y) &\text{if $y=\frac{1}{3}, x\in \cint{\frac{2}{3},1},$}
    \\U(x,y) &\text{otherwise,} \end{cases}$$
   then $V$ is also a uninorm. Here  $V$ is non-continuous in the point $(\frac{1}{3},\frac{2}{3}),$
   however, both $v_{\frac{1}{3}}$ and $v_{\frac{2}{3}}$ are continuous.
   Note that $(\uint,V)$ is an ordinal sum of semigroups $G_{\alpha},G_{\gamma}$ and  $G_{\beta^*}= (\rint{\frac{1}{3},\frac{2}{3}},T_2^*),$
   $G_{\delta} = (\{\frac{1}{3}\},T_2^*),$ where  $\alpha < \delta< \gamma < \beta^*.$
    \label{exext}
\end{example}

\begin{figure}[htb] \begin{center}
\begin{picture}(150,150)
\put(0,0){ \line(1,0){150} }
\put(0,0){ \line(0,1){150} }
\put(150,150){ \line(-1,0){150} }
\put(150,150){ \line(0,-1){150} }

\put(0,50){ \line(1,0){150} }
\put(50,0){ \line(0,1){150} }
\put(50,100){ \line(1,0){100} }
\put(100,50){ \line(0,1){100} }
\put(20,20){$T_1^*$}
\put(20,95){$\min$}
\put(95,20){$\min$}
\put(70,120){$\max$}
\put(120,70){$\max$}
\put(70,70){$T_2^*$}
\put(120,120){$S^*$}

\linethickness{0.5mm}
\put(50,100){ \line(0,1){50} }
\put(50,100){ \line(1,0){48} }
\put(100,50){ \line(0,1){48} }
\put(100,50){ \line(1,0){50} }

\put(104,100){\circle{4}}

\end{picture}
\end{center}
\caption{The uninorm $U$ from Example \ref{exext}. The bold  lines denote the points of discontinuity of $U.$} \label{figex}
\end{figure}

%The following result follows easily from the monotonicity, Proposition \ref{procon} and Proposition \ref{atmost}.

%\begin{lemma}  Let $U\colon \uint^2\lran \uint$ be a uninorm, $U\in \mathcal{U}.$ Assume $(x_0,y_0)\in \uint^2.$
%If there exist $\varepsilon_1>0$ and $\varepsilon_2>0$ such that $u_z$ is non-continuous in $x_0$ and $u_v$ is non-continuous in $y_0$
%  for all $z\in \rint{y_0,y_0+\varepsilon_1},$ $v\in \rint{x_0,x_0+\varepsilon_2}$  (for all  $z\in \lint{y_0-\varepsilon_1,y_0},$ $v\in %\lint{x_0+\varepsilon_2,x_0}$) then $U$ is non-continuous in $(x_0,y_0).$ \label{lemroh}
%\end{lemma}

The following two results show that the set of discontinuity points of a uninorm $U\in \mathcal{U}$ from  the set
$\cint{0,e}\times \cint{e,1}$ ($\cint{e,1}\times \cint{0,e}$) is connected.

\begin{proposition} Let $U\colon \uint^2\lran \uint$ be a uninorm, $U\in \mathcal{U}.$
Let $u_{x_1}$ be non-continuous in $y_1$ and $u_{x_2}$ be non-continuous in $y_2$ for $x_1<x_2\leq e$ ($e\leq x_1<x_2$). Then for all
 $y\in \cint{y_2,y_1}$ either there exists $x^*\in \cint{x_1,x_2} $ such that $u_{x^*}$ is non-continuous in $y$ or there is an interval $\cint{c,d},$ where $y\in \cint{c,d}\subset \uint,$ and $p \in \cint{x_1,x_2}$ such that
 $u_z$ is non-continuous in $p$ for all $z\in \cint{c,d} .$
\end{proposition}

\begin{proof}
If $u_{x_1}$ is non-continuous in $y_1$ and $u_{x_2}$ is non-continuous in $y_2$ for $x_1<x_2\leq e$ (the case when $e\leq x_1<x_2$ is analogous) then   $U(x_2,z)<e$ for all $z<y_2$ and $U(x_1,z)>e$ for all $z>y_1$ and the monotonicity
  implies that for all  $x\in \cint{x_1,x_2} $
the function $u_{x}$ is non-continuous in some point  $z\in \cint{y_2,y_1}.$
\textcolor{red}{Note that $e\notin\mathrm{Ran}(u_x)$ since otherwise by Proposition \ref{procon} either $u_{x_1}$ or $u_{x_2}$ would be continuous.}
 Assume the function $g\colon \cint{x_1,x_2}\lran \cint{y_2,y_1}$
which assigns to $v\in \cint{x_1,x_2}$ a point $w\in \cint{y_2,y_1}$ such that $u_v$ is non-continuous in $w.$
 Then by Proposition \ref{promon} the function $g$ is non-increasing. If $q\in \cint{y_2,y_1}\setminus \mathrm{Ran}(g)$ then by the monotonicity there exists a $p\in \cint{x_1,x_2}$ such that $g(d)>q$ if $d<p$ and $g(d)<q$ if $d>p.$ Further, since $g$ is monotone there exists an interval $\cint{c,d},$ such that $q\in \cint{c,d}\subset \cint{y_2,y_1}\setminus \mathrm{Ran}(g).$ Then for $z\in \cint{c,d}$ we have $U(z,v)<e$ for all $v<p$ and  $U(z,v)>e$ for all $v>p$ thus $u_z$ has a point of discontinuity in
 $p.$
\end{proof}

\begin{lemma}
Let $U\colon \uint^2\lran \uint$ be a uninorm, $U\in \mathcal{U}.$ Let $u_{x}$ be non-continuous in $y_1$ and $u_{y_2}$ be non-continuous in $x$ for some $y_1\neq y_2.$ Then for all $y\in \rint{y_1,y_2}$ ($y\in \lint{y_2,y_1}$) the function $u_y$ is non-continuous in $x.$ \label{lemconv}
\end{lemma}

\begin{proof}
We will assume $y_1<y_2$ (the case when $y_1>y_2$ is analogous). Then $U(x,y)>e$ for all $y>y_1$ and $U(z,y)\leq U(z,y_2)<e$ for all $z<x,$
$y\leq y_2.$ Then since $U(x,y)\neq e$ the function $u_y$ is non-continuous in $x.$
\end{proof}

In the following result we show that the set of discontinuity points of a uninorm $U\in \mathcal{U}$ has a non-empty intersection with the border of the unit square.

\begin{lemma}
Let $U\colon \uint^2\lran \uint$ be a uninorm, $U\in \mathcal{U}.$ Assume  $x<e$ ($x>e$) such that  $u_{x}$  is continuous on $\uint$ and let $u_{y}$ be non-continuous in $x.$ Then for all $q\in \cint{y,1}$  ($q\in \cint{0,y}$) the function $u_q$ is non-continuous in $x.$ \label{lemconv2}
\end{lemma}

\begin{proof} We will assume $x<e$ (the case for $x>e$ is analogous).
If $U(x,z)=e$ for some $z\in \uint $ then by Lemma \ref{noid} the points $x,z$ are not idempotents and Proposition \ref{procon} implies that $U$ is continuous on
$\uint^2\setminus (\lint{0,a}\cup\rint{b,1})^2$ for some $a<x$ and $b>z.$
Therefore for all $y\in \uint$ the function  $u_{y}$ is continuous in $x.$ Since $x<e $ by Lemma \ref{lemconux} we have  $u_x(1)<e,$  i.e., $U(x,z)<e$ for all $z\in \uint.$
If $u_{y}$ is non-continuous in $x$ then  $U(p,y)>e$ for all $p>x$ and $U(p,y)<e$ for all $p<x.$
Assume any $q\in \cint{y,1}.$ Then $U(p,q)\leq U(x,q) <e$ if $p<x$ and  $U(p,q)\geq U(p,y)>e$ if $p>x,$ i.e.,$ u_q$ is non-continuous in $x.$
\end{proof}

Next we define a set-valued function.

\begin{definition}
A mapping $p\colon \uint\lran \mathcal{P}(\uint)$ is called a set-valued function on $\uint$ if to every $x\in \uint$ it assigns a subset of $\uint,$ i.e.,
$p(x)\subseteq \uint.$ Assuming the standard order on $\uint,$ a set-valued function $p$ is called
\begin{mylist}
\item \emph{non-increasing} if  for all $x_1,x_2\in \uint,$ $x_1<x_2,$ we have
$y_1\geq y_2$ for all $y_1\in p(x_1)$ and all $y_2\in p(x_2)$  and thus the cardinality $\mathrm{Card}(p(x_1)\cap p(x_2))\leq 1,$
\item  \emph{symmetric} if $y\in p(x)$ if and only if $x\in p(y).$
\end{mylist}
 The graph of a set-valued function $p$ will be denoted by $G(p),$ i.e., $(x,y)\in G(p)$ if and only if
$y\in p(x).$
\end{definition}

The following is evident.

\begin{lemma}
A symmetric set-valued function $p\colon \uint\lran$ $\mathcal{P}(\uint)$ is surjective, i.e.,  for all $y\in \uint$ there exists an $x\in \uint$ such that $y\in p(x),$ if and only if we have $p(x)\neq \emptyset$ for all $x\in \uint.$
 \label{lemnone}
\end{lemma}

 The graph of a symmetric, surjective,  non-increasing  set-valued function $p\colon \uint\lran$ $\mathcal{P}(\uint)$ is a connected line (i.e., a connected set with no interior) containing points $(0,1)$ and $(1,0)$.
Indeed, the monotonicity of such a set-valued function ensures that the graph of $p$ has no interior. Further, since $p$ is surjective, monotone and symmetric the graph of $p$ contains points $(0,1)$ and $(1,0).$ If $G(p)$ is not a connected set then either $p(x)$ is not a connected set for some $x\in \uint,$ which, however, due to the monotonicity implies  that $p$ is not surjective, or  due to the monotonicity there exists an $x\in \uint$ such that either
$$ \inf(\bigcup\limits_{q<x} p(q)) > \sup(p(x)),$$  or $$ \sup(\bigcup\limits_{q>x} p(q)) < \inf(p(x)),$$  which, however,  due to the symmetry  implies that $p$ is not surjective.

The previous results can be summarized in the following theorem. First, however, we introduce one remark.

\textcolor{red}{
\begin{remark}
For any uninorm $U\colon \uint^2\lran \uint,$  $U\in \mathcal{U}$ denote $A=\inf\{x\mid U(x,0)>0\},$  $B=\sup\{x\mid U(x,1)<1\}$ and let $a,d\in \uint$ be such that $U(x,y)=e$ for some $y\in \uint$ if and only if $x\in \opint{a,d}\cup \{e\}$ (see Remark \ref{remad}). If $U$ is conjunctive, i.e., $U(0,1)=0,$ 
 then $A$ is the infimum of an empty set on $\uint,$ i.e., $A=1.$ 
If $U$ is disjunctive, i.e., $U(0,1)=1,$
 then $B$ is the supremum of an empty set on $\uint,$ i.e., $B=0.$ 
   Therefore we have
 either $A=1, B\neq 0,$ or $A\neq 1,B=0,$ or $A=1, B= 0.$
 Further,   $U(x,0)\leq e$ for some $x\in \uint$ implies $$0=U(e,0)\geq U(x,0,0)=U(x,0) $$ and thus for all $x\in \uint$ either $U(x,0)=0$
or $U(x,0)>e.$ Therefore $U$ is non-continuous in $(0,A)$ if $A\neq 1.$ Similarly, 
 $U(x,1)\geq e$ for some $x\in \uint$ implies $$  1 =U(e,1)\leq U(x,1,1)=U(x,1) $$ and thus for all  $x\in \uint$ either $U(x,1)=1$
or $U(x,1)<e.$ Therefore $U$ is non-continuous in $(B,1)$ if $B\neq 0.$
Finally, if $A=1, B=0$ then $U(x,0)=0$ for all $x<1$ and $U(x,1)=1$ for all $x>0$ and therefore $U$ is non-continuous in $(0,1).$ 
 \newline Due to Remark \ref{remad} either $a=d=e,$ or
  $U$ is continuous on $\opint{a,d}\times \uint \cup \uint \times \opint{a,d}$ 
   and therefore we  have $0\leq B \leq a \leq e \leq d \leq A \leq 1.$
\label{pomfun}
\end{remark}}

\begin{theorem}
Let $U\colon \uint^2\lran \uint$ be a  uninorm, $U\in \mathcal{U}.$ Then there exists a symmetric,  surjective, non-increasing  set-valued function $r$ on $\uint$ such that $U$ is continuous on
$\uint^2\setminus R,$ where $R=G(r).$ Note that $U$ need not to be non-continuous in all points from $R.$ \label{mufu}
\end{theorem}

\begin{proof}
We will define the set $R^*=\{(x,y)\in \uint^2 \mid \text{ $U$ is non-continuous in $(x,y)$}\}.$ Then due to the commutativity of $U$
the set $R^*$ is symmetric, i.e., $(x,y)\in R^*$ if and only if $(y,x)\in R^*.$
If we   define a set-valued function $r\colon \uint \lran $ $\mathcal{P}(\uint)$ by
\begin{equation} r(x)=\begin{cases}
\{1\} &\text{if $x\in\opint{0,B},$ } \\
 \{0\} &\text{if $x\in\opint{A,1},$ } \\
 \cint{0,B} &\text{if $x=1,$ } \\
 \cint{A,1} &\text{if $x=0,$ } \\
 \{y\mid U(x,y)=e\}  &\text{if $x\in\opint{a,d}\cup \{e\},$ } \\
 \{y\mid (x,y)\in R^*\} &\text{otherwise} \end{cases}\label{eqr} \end{equation}
then $r$ is a symmetric set-valued function (see Remark \ref{pomfun}).
Since $u_x$ is continuous if and only if $x\in \lint{0,B}\cup \opint{a,d}\cup \{e\} \cup \rint{A,1}$  (which follows from Lemma  \ref{lemconux} and Proposition \ref{procon}) Lemma \ref{lemnone} implies that
$r$ is surjective. Moreover, it is evident that if $U$ is non-continuous in $(x_0,y_0)$ then $x_0\in r(y_0).$

We will further define the set $$P=\{(x,y)\in \uint^2\mid \text{$u_x$ is continuous in $y$ and $u_y$ is continuous in $x$}\}.$$
Assume $x_1<x_2$ and  $y_1\in r(x_1),$   $y_2\in r(x_2).$

\noindent {\bf Case $1$}:
If $(x_1,y_1),(x_2,y_2)\in R\setminus P$  then Proposition \ref{promon} implies $y_1\geq y_2.$

\noindent {\bf Case $2$}:
Assume $(x_1,y_1) \in  P\cap R,$ $(x_2,y_2)\in R\setminus P.$ Then Proposition   \ref{prononco} implies that either
$(x_3,y_1)\in R\setminus P$ for some $x_3\in \uint,$ $x_1<x_3<x_2,$ or $(x_1,y_3)\in R\setminus P$ for some $y_3,$ $y_3<y_1.$
Now since $(x_2,y_2)\in R\setminus P$ the case when $(x_3,y_1)\in R\setminus P$ implies by Proposition \ref{promon}
$y_1\geq y_2.$ In the case when $(x_1,y_3)\in R\setminus P$ we have $y_1>y_3\geq y_2.$

\noindent {\bf Case $3$}:
Assume  $(x_2,y_2) \in  P\cap R$ and $(x_1,y_1)\in R\setminus P.$ This case can be shown similarly as  the Case $2.$

\noindent {\bf Case $4$}:
Assume  $(x_1,y_1),(x_2,y_2)\in P\cap R.$ Then Proposition   \ref{prononco} implies that either
$(x_4,y_2)\in R\setminus P$ for some $x_4\in \uint,$ $x_3<x_4<x_2$ ($x_1<x_4<x_2$), or $(x_2,y_4)\in R\setminus P$ for some $y_4,$ $y_4>y_2.$
If $(x_3,y_1)\in R\setminus P$ and $(x_4,y_2)\in R\setminus P$ we have $y_1\geq y_2.$ If $(x_3,y_1)\in R\setminus P$
and $(x_2,y_4)\in R\setminus P$ we have $y_1\geq y_4>y_2.$ If   $(x_1,y_3)\in R\setminus P$ and  $(x_4,y_2)\in R\setminus P$ we have $y_1>y_3\geq y_2.$ If  $(x_1,y_3)\in R\setminus P$ and  $(x_2,y_4)\in R\setminus P$ we have $y_1>y_3\geq y_4>y_2.$

\

\noindent Therefore in all cases $y_1\geq y_2$ and thus we have shown that $r$ is non-increasing on $\cint{B,a}\cup \cint{d,A}.$ Since $r$ is evidently non-increasing also on $\lint{0,B}\cup \opint{a,d} \cup \{e\} \cup \rint{A,1}$ we see that  $r$ is non-increasing.

 \end{proof}

\begin{remark}
$U$ need not to be non-continuous in all points of $R.$
From the previous proof we see that $U$ is continuous in all points from  $\{x\}\times \uint$ for all $x\in \lint{0,B}\cup \opint{a,d} \cup \{e\} \cup \rint{A,1}.$
The symmetric   non-increasing  set-valued function from the previous theorem need not to be unique. The differences can appear on $\opint{a,d}^2.$
However, if we require additionally that $U(x,y)=e$ implies $(x,y)\in R$ for all $(x,y)\in \uint^2,$ such a set-valued function is uniquely given and we will call such a  set-valued function the \emph{characterizing} set-valued function of a uninorm $U$ for $U\in \mathcal{U}.$
 \label{remthree}
\end{remark}

\begin{example}
Assume a representable uninorm  $U_1\colon \uint^2\lran \uint$  and a continuous t-norm $T\colon \uint^2\lran \uint$ and a continuous t-conorm
$S\colon \uint^2\lran \uint.$ For simplicity we will assume that $\frac{1}{2}$ is the neutral element of $U_1$ and that $U_1(x,1-x)=\frac{1}{2}$ for all
$x\in \opint{0,1}.$
Let  $U_1^*$ be a linear transformation of $U_1$ to $\cint{\frac{1}{3},\frac{2}{3}}^2,$
let $T^*$ be a linear transformation of $T$ to $\cint{0,\frac{1}{3}}^2$ and let $S^*$
be a linear transformation of $S$ to $\cint{\frac{2}{3},1}^2.$
Then the ordinal sum of semigroups $G_{\alpha}= (\cint{0,\frac{1}{3}},T^*),$
$G_{\beta}= (\cint{\frac{1}{3},\frac{2}{3}},U_1^*),$
$G_{\gamma}= (\cint{\frac{2}{3},1},S^*),$ with $ \gamma < \alpha  < \beta ,$
is a semigroup $(\uint,U),$ where $U$ is a uninorm, $U \in \mathcal{U}.$
 On Figure \ref{figcom} we can see the characterizing set-valued function $r$ of $U$ as well as its set of discontinuity points.
\label{exla}
\end{example}

\begin{figure}[htb] \hskip2cm
\begin{picture}(150,150)
\put(0,0){ \line(1,0){150} }
\put(0,0){ \line(0,1){150} }
\put(150,150){ \line(-1,0){150} }
\put(150,150){ \line(0,-1){150} }

\put(0,50){ \line(1,0){100} }
\put(50,0){ \line(0,1){100} }
\put(0,100){ \line(1,0){150} }
\put(100,0){ \line(0,1){150} }

\put(70,70){$U_1^*$}
\put(25,20){$T^*$}
\put(125,120){$S^*$}
\put(50,120){$\max$}
\put(120,50){$\max$}
\put(70,20){$\min$}
\put(20,70){$\min$}

\linethickness{0.5mm}
\put(0,100){ \line(1,0){50} }

\put(100,0){ \line(0,1){50} }
\end{picture}
 \hskip2cm
\begin{picture}(150,150)
\put(0,0){ \line(1,0){150} }
\put(0,0){ \line(0,1){150} }
\put(150,150){ \line(-1,0){150} }
\put(150,150){ \line(0,-1){150} }

\put(0,50){ \line(1,0){100} }
\put(50,0){ \line(0,1){100} }
\put(0,100){ \line(1,0){150} }
\put(100,0){ \line(0,1){150} }

\put(50,100){ \line(1,-1){50} }

\put(62,62){$U_1^*$}
\put(82,82){$U_1^*$}
\put(25,20){$T^*$}
\put(125,120){$S^*$}
\put(50,120){$\max$}
\put(120,50){$\max$}
\put(70,20){$\min$}
\put(20,70){$\min$}

\linethickness{0.5mm}
\put(0,100){ \line(1,0){50} }
\put(0,100){ \line(0,1){50} }
\put(100,0){ \line(1,0){50} }
\put(100,0){ \line(0,1){50} }
\end{picture}
\caption{The uninorm $U$ from Example \ref{exla}. Left: the bold lines denote the points of discontinuity of $U.$
Right: the oblique and bold lines denote the characterizing set-valued function of $U.$}
\label{figcom}
\end{figure}
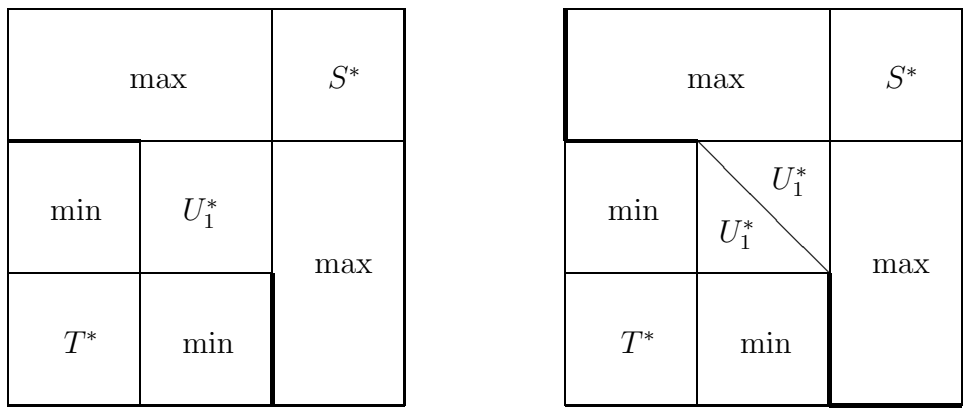

\begin{remark}
It is easy to see that for $U\in \mathcal{U}$ its characterizing set-valued function $r$ divides the uninorm $U$ into two parts:
$U$ on points below the characterizing set-valued function attains values smaller than $e,$ and $U$ on points above the characterizing set-valued function attains values bigger than $e.$
\end{remark}

\begin{proposition}
Let $U\colon \uint^2\lran \uint$ be a  uninorm, $U\in \mathcal{U}.$  Then in each point $(x_0,y_0)\in \uint^2$ the uninorm $U$ is either left-continuous or right continuous.
\end{proposition}

\begin{proof}
From Proposition \ref{rl} we know that for all $x\in \uint$ the function $u_x$ is either left-continuous or right continuous.
If $(x_0,y_0)$ is the point of continuity of $U$ the claim is trivial.
 Thus suppose that $(x_0,y_0)$ belongs to the graph of the characterizing set-valued function $r$ of $U.$
 If $U(x_0,y_0)=e$ then by Proposition \ref{procon} the uninorm $U$ is continuous in $(x_0,y_0)$ and thus either
 $U(x_0,y_0)<e$ or $U(x_0,y_0)>e.$
If $U(x_0,y_0)<e$ then for all $x\leq x_0,$ $y\leq y_0$
also $U(x,y)<e$  and thus $u_x$ is left-continuous in $y$ and $u_y$ is left-continuous in $x$ (see Remark \ref{remlrint}).
Now for any $\varepsilon>0$ there exists $\delta_1>0$ such that
$\vert U(x_0-\delta_1,y_0) - U(x_0,y_0) \vert <\frac{\varepsilon}{2}.$
Since also $u_{x_0-\delta_1}$ is left-continuous in $y_0$ there exists $\delta_2>0$ such that
$\vert U(x_0-\delta_1,y_0-\delta_2) - U(x_0-\delta_1,y_0) \vert<\frac{\varepsilon}{2}.$
The monotonicity of $U$ then implies that
\begin{multline*} 0 \leq   U(x_0,y_0) -  U(x_0-\delta_1,y_0-\delta_2) = \\  U(x_0,y_0)  - U(x_0-\delta_1,y_0) +  U(x_0-\delta_1,y_0)-  U(x_0-\delta_1,y_0-\delta_2)
< \varepsilon.\end{multline*} Taking $\delta = \min(\delta_1,\delta_2),$ by the monotonicity of $U$ we have shown that for each $\varepsilon>0$ there exists a $\delta>0$
such that if $x\in \cint{x_0-\delta,x_0}$ and $y\in \cint{y_0-\delta,y_0}$
 we have
 $\vert U(x,y)-U(x_0,y_0)\vert<\varepsilon,$
 i.e., that
  $U$ is left-continuous in $(x_0,y_0).$ Similarly,
if  $U(x_0,y_0)>e$ then $U$ is right-continuous in $(x_0,y_0).$
\end{proof}

The previous proposition and the construction of the characterizing set-valued function $r$ of a uninorm $U$ implies the following.

\begin{corollary}
Let $U\colon \uint^2\lran \uint$ be a  uninorm, $U\in \mathcal{U}.$ Then there exists a symmetric,  surjective, non-increasing  set-valued function $r$ on $\uint$ such that $U$ is continuous on
$\uint^2\setminus R,$ where $R=G(r)$ and if $U(x,y)=e$ then $(x,y)\in R.$
 Moreover, in each point $(x,y)\in \uint^2$ the uninorm $U$ is either left-continuous or right-continuous.
\end{corollary}

%\begin{remark}
%The graph of a  symmetric,  surjective, non-increasing  set-valued function can be divided into  connected maximal segments which are either strictly decreasing, or the horizontal segments, or the vertical segments. Note that a horizontal segment corresponds to a closed interval $Z$ such that there exists a $y\in \uint$ with $ r(x)=\{y\}$ for all $x\in\mathrm{int}(Z)$ (where $\mathrm{int}(Z)$ is the interval $Z$ without border points) and a vertical segment corresponds to a closed interval $V$ such that $ r(x)=V,$ $\mathrm{Card}(V)>1,$ for some $x\in \uint.$  We say that segment corresponding to an interval $S$ is strictly decreasing if  $y_1 \in r(x_1),$ $y_2\in r(x_2)$ for
%$x_1,x_2\in S,$ $x_1<x_2$ implies $y_2<y_1$ and $\mathrm{Card}(r(x))=1$ for all $x\in \mathrm{int}(S).$ The previous description implies that all horizontal, vertical and strictly decreasing  segments correspond to closed intervals. \label{remseg}
%\end{remark}

\subsection{The sufficiency part}
\label{sec3.2}

In this part we will show that if for a uninorm $U$ there exists a symmetric,  surjective, non-increasing set-valued function $r$ on
$\uint$ such that $U$ is continuous on
$\uint^2\setminus R,$ where $R=G(r),$ and  $U(x,y)=e$ implies $(x,y)\in R,$ then    $U\in \mathcal{U}$ if and only if in each point $(x,y)\in \uint^2$ the uninorm $U$ is either left-continuous or right-continuous.

We will denote the set of all uninorms  $U\colon \uint^2\lran \uint$ such that  $U$ is continuous on
$\uint^2\setminus R,$ where $R=G(r)$ and $r$ is a symmetric,  surjective, non-increasing set-valued function such that  $U(x,y)=e$ implies $(x,y)\in R,$
by $\mathcal{UR}.$
First we will show that there exists a uninorm $U\in \mathcal{UR}$ such that $U\notin \mathcal{U}$.

\begin{example}
Let $U\colon \uint^2\lran \uint$ be   given by
$$U(x,y)=\begin{cases} 0 &\text{if $\max(x,y)<e,$} \\
x &\text{if $y=e,$} \\
y &\text{if $x=e,$} \\
\max(x,y) &\text{otherwise.} \end{cases}$$
Then Proposition \ref{mmset} implies that $U\in \mathcal{U}_{\max}$ is a uninorm, where the underlying t-norm is the drastic product and
the underlying t-conorm is the maximum. This uninorm is non-continuous in points from $\{e\}\times \cint{0,e} \cup \cint{0,e} \times \{e\}.$
Thus the corresponding set-valued function is given by (see Figure \ref{Figdra})
$$r(x) = \begin{cases} \cint{e,1} &\text{if $x=0,$} \\
e &\text{if $x\in \opint{0,e},$} \\
\cint{0,e} &\text{if $x=e,$} \\
0 &\text{otherwise.} \end{cases}$$ Since $U(x,y)=e$ implies $x=y=e$
we see that $U$ is continuous on
$\uint^2\setminus R,$ where $R=G(r)$ and $r$ is a symmetric,  surjective, non-increasing set-valued function such that  $U(x,y)=e$ implies $(x,y)\in R.$
However, the drastic product t-norm is not continuous and thus  $U\notin \mathcal{U}.$ \label{exdra}
\end{example}

\begin{figure}[htb]  \begin{center}
\begin{picture}(150,150)
\put(0,0){ \line(1,0){150} }
\put(0,0){ \line(0,1){150} }
\put(150,150){ \line(-1,0){150} }
\put(150,150){ \line(0,-1){150} }
\put(0,75){ \line(1,0){150} }
\put(75,0){ \line(0,1){150} }
\linethickness{0.5mm}
\put(0,75){ \line(1,0){75} }
\put(0,75){ \line(0,1){75} }
\put(75,0){ \line(0,1){75} }
\put(75,0){ \line(1,0){75} }
\put(30,30){$0$}
\put(110,110){$\max$}
\put(110,30){$\max$}
\put(30,110){$\max$}
\end{picture}
\end{center}
\caption{The uninorm $U$  from Example \ref{exdra}. The bold lines denote the characterizing set-valued function $r$ of $U.$} \label{Figdra}
\end{figure}

Assume $U\in \mathcal{UR}.$ Then for the corresponding characterizing set-valued function $r$ we have $(e,e)\in G(r).$ Denote
$$D=\{e\}\times \uint \cup \uint\times\{e\}.$$
We have two possibilities: either $G(r)\cap D=\{(e,e)\},$ or $\mathrm{Card}(G(r)\cap D)>1.$ First we will assume the case when $G(r)\cap D=\{(e,e)\}.$
Then $T_U$  ($S_U$) is continuous in all points from $\cint{0,e}^2$ ($\cint{e,1}^2$) except possibly the point $(e,e)$ and we have the following result.

\begin{lemma}
Let $T\colon \uint^2\lran \uint$ be a t-norm which is continuous on $\uint^2 \setminus \{(1,1)\}.$ Then $T$ is continuous on $\uint^2.$
\label{onepo}
\end{lemma}

\begin{proof} Assume that $T$ is not continuous in $(1,1).$ Then there exist
 two sequences $\{a_n\}_{n\in \mathbb{N}},$ $a_n\in \opint{0,1}$ and $\{b_n\}_{n\in \mathbb{N}},$ $b_n\in \opint{0,1}$  such that
$\lim\limits_{n\lran \infty} a_n = \lim\limits_{n\lran \infty} b_n =1 $ and
$\lim\limits_{n\lran \infty}  T(a_n,b_n)<1.$ Since $T(a_n,b_n)\geq T(\min(a_n,b_n),\min(a_n,b_n))$ we see that there exists a sequence
$\{c_n\}_{n\in \mathbb{N}},$ $c_n \in \opint{0,1},$ $\lim\limits_{n\lran \infty} c_n =1 $ such that $\lim\limits_{n\lran \infty}  T(c_n,c_n)=1-\delta<1,$ for some
$\delta>0.$ Since $T$ is a t-norm we have $T(1-\frac{\delta}{2},1) = 1-\frac{\delta}{2}$ and necessarily
$T(1-\frac{\delta}{2},1-\frac{\delta}{2})\leq  1-\delta.$ Since $T$ is continuous on $\uint^2 \setminus \{(1,1)\}$ there exists
an $\varepsilon>0$   such that $T(1-\frac{\delta}{2},1-\varepsilon)=1-\frac{2\delta}{3}$ and the monotonicity of $T$ implies
$\varepsilon<\frac{\delta}{2}.$ Thus $$1-\frac{2\delta}{3}=T(1-\frac{\delta}{2},1-\varepsilon)\leq T(1-\varepsilon,1-\varepsilon)\leq 1-\delta,$$
what is a contradiction.
\end{proof}

By duality between t-norms and t-conorms we get the following.

\begin{lemma}
Let $S\colon \uint^2\lran \uint$ be a t-conorm which is continuous on $\uint^2 \setminus \{(0,0)\}.$ Then $S$ is continuous on $\uint^2.$
\label{onepo1}
\end{lemma}

From the two previous results we see that if $U\in \mathcal{UR}$ and $G(r)\cap D=\{(e,e)\}$
 then  $U\in \mathcal{U}.$

Further we will suppose that $\mathrm{Card}(G(r)\cap D)>1.$
Then we obtain the following result.

\begin{lemma}
Let $U\colon \uint^2\lran \uint$ be a uninorm, $U\in \mathcal{UR},$ $U\notin \mathcal{U}.$
Then there exists a point $(x,y)\in \uint^2$ such that $U$ is neither left-continuous, nor right-continuous in $(x,y).$
\end{lemma}

\begin{proof} Since $U\notin \mathcal{U}$ Lemmas \ref{onepo} and \ref{onepo1} imply that $\mathrm{Card}(G(r)\cap D)>1.$
Then there exists an $x_1\in \uint,$ $x_1\neq e$ such that $(x_1,e)\in G(r).$ We will suppose that $x_1<e$ (the case when $x_1>e$ is analogous).
Let $$x_0=\inf\{x\in \cint{0,e}\mid (x,e)\in G(r)\}.$$
Then the monotonicity of $r$  implies that   $S_U$ is continuous and
$\rint{x_0,e}\times \{e\} \subset G(r).$ Moreover, $U(x,y)=e$ implies $x=y=e$ for all $x,y\in \uint.$
Since $U$ is continuous on $\rint{x_0,e} \times \rint{e,1} \cup \rint{e,1} \times \rint{x_0,e}$
we see  that $U(x,y)>e$ for all $x\in \rint{x_0,e},$ $y\in \rint{e,1}.$
 On the other hand, the neutral element $e$ and the monotonicity of
$U$ implies $U(x,y)\in \cint{x,y}$ for all $x\in \rint{x_0,e},$ $y\in \rint{e,1}.$
Thus for all $x\in \opint{x_0,e}$ we have $\lim\limits_{s\lran e^+} U(x,s)=e.$ Therefore on $\opint{x_0,e}$ the uninorm $U$ is not right-continuous.
Since $U\notin \mathcal{U}$ and $T_U$  is continuous on $\lint{0,1}^2$ we see that $U$ is not left-continuous in some point $(x,e)$ for
$x\in \cint{x_0,e}.$ Now similarly as in Lemma \ref{onepo} we can show that  $U$ is not left-continuous in some point $(x,e)$ for
$x\in \lint{x_0,e}.$ Finally, the neutral element and the monotonicity of $U$ imply that $U$ is not left-continuous in some point $(x,e)$ for
$x\in \opint{x_0,e}.$ Summarising,  there exists a point $(x,y)\in \uint^2$ such that $U$ is neither left-continuous, nor right-continuous in $(x,y).$
\end{proof}

All previous results can be compiled into the following theorem.

\begin{theorem}
 Let $U\colon \uint^2\lran \uint$ be a uninorm, $U\in \mathcal{UR}.$ Then    $U\in \mathcal{U}$ if and only if in each point $(x,y)\in \uint^2$ the uninorm $U$ is either left-continuous or right-continuous. \label{thco}
\end{theorem}

\begin{corollary}
Let $U\colon \uint^2\lran \uint$ be a uninorm. Then $U\in \mathcal{U}$ if and only if $U\in \mathcal{UR}$ and  in each point $(x,y)\in \uint^2$ the uninorm $U$ is either left-continuous or right-continuous.
\end{corollary}

\section{Conclusions}
\label{sec5}

 We have shown that  a uninorm with continuous underlying t-norm and t-conorm is
continuous on $\uint^2\setminus R,$ where $R$ is the graph of some symmetric, surjective,  non-increasing set-valued function.
On the other hand, we have shown also a sufficient condition for a uninorm to have continuous underlying operations.
 In the follow up papers \cite{extord,supja} we will employ these results and using
 the characterizing set-valued function of a uninorm we will show that each uninorm with continuous underlying t-norm and t-conorm can be decomposed into an
 ordinal sum of semigroups related to representable uninorms, continuous Archimedean t-norms, continuous Archimedean t-conorms, internal uninorms and singleton semigroups. Thus these three papers together offer a complete characterization of  uninorms from $\mathcal{U},$ i.e., of uninorms with continuous underlying t-norm and t-conorm.  The applications of these results are expected in all domains where uninorms are used.

 \vskip0.3cm \noindent{\bf \large Acknowledgement} \hskip0.3cm  This work was supported by grants  VEGA 2/0049/14, APVV-0178-11
 and Program Fellowship of SAS.

\end{document}